\documentclass[reqno]{amsart}

\input xy
\xyoption{all}
\usepackage{accents}
\usepackage{epsfig}
\usepackage{color}
\usepackage{amsthm}
\usepackage{amssymb}
\usepackage{amsmath}
\usepackage{amscd}
\usepackage{amsopn}
\usepackage{graphicx}
\usepackage{centernot}

\usepackage{xspace}

\usepackage{url}
\usepackage{enumitem, hyperref}\hypersetup{colorlinks}

\usepackage{color} %\textcolor{red}{Test}

\definecolor{darkred}{rgb}{1,0,0} %can change the intensity in [0,1]
\definecolor{darkgreen}{rgb}{0,0.6,0}
\definecolor{darkblue}{rgb}{0,0,.8}

\hypersetup{colorlinks,
linkcolor=darkblue,
filecolor=darkgreen,
urlcolor=darkred,
citecolor=darkgreen}

\makeatletter
\def\reflb#1#2{\begingroup
    #2%
    \def\@currentlabel{#2}%
    \phantomsection\label{#1}\endgroup
}
\makeatother

\numberwithin{equation}{section}
\newtheorem {Theorem}{Theorem}
\numberwithin{Theorem}{section}

\newtheorem {Lemma}[Theorem]    {Lemma}

\theoremstyle{definition}

\theoremstyle{remark}
\newtheorem{Remark}[Theorem]{Remark}

\def    \F      {{\mathbb F}}

\def    \C      {{\mathbb C}}
\def    \R      {{\mathbb R}}

\def    \Z      {{\mathbb Z}}

\def    \12    {{\frac{1}{2}}}

\def    \im     {\operatorname{im}}

\def    \GL     {\operatorname{GL}}

\begin{document}

%%%%%%%%%%%%%%%%%%%%%%%%%%%%%%
%   TEXT FORMATTING

\setlength{\smallskipamount}{6pt}
\setlength{\medskipamount}{10pt}
\setlength{\bigskipamount}{16pt}

%%%%%%%%%%%%%%%%%%%%%%%%%%

%%%%%%%%%%%%%%%%%%%%%%%%%%

%%%%%%%%%%%           BEGINNING OF  TEXT

%%%%%%%%%%%%%%%%%%%%%%%%%%

\title[Random Chain Complexes]{Random Chain Complexes}

\author[Viktor Ginzburg]{Viktor L. Ginzburg}
\author[Dmitrii Pasechnik]{Dmitrii V. Pasechnik}

\address{DP: Department of Computer Science, University of Oxford,
Wolfson Building, Parks Road, Oxford, OX1 3QD, UK} \email{dmitrii.pasechnik@cs.ox.ac.uk}

\address{VG: Department of Mathematics, UC Santa Cruz, Santa Cruz, CA
  95064, USA} \email{ginzburg@ucsc.edu}



\date{\today} 

\thanks{The work is partially supported by NSF grant DMS-1308501 (VG)
  and the EU Horizon 2020 research and innovation programme, grant
  agreement OpenDreamKit No 676541 (DP)}

\bigskip

\begin{abstract}
  We study random, finite-dimensional, ungraded chain complexes
  over a finite field and show that for a uniformly distributed
  differential a complex has the smallest possible homology with the
  highest probability: either zero or one-dimensional homology depending on the
  parity of the dimension of the complex. We prove that as the order
  of the field goes to infinity the probability distribution
  concentrates in the smallest possible dimension of the homology. On
  the other hand, the limit probability distribution, as the dimension
  of the complex goes to infinity, is a super-exponentially decreasing, but
  strictly positive, function of the dimension of the homology.
\end{abstract}

\maketitle

%\tableofcontents

\section{Introduction}
\label{sec:intro} 
We study random, finite-dimensional, ungraded chain complexes over a finite
field and we are interested in the probability that such a complex has
homology of a given dimension. We show that for a uniformly
distributed differential the complex has the smallest possible
homology with the highest probability.

To be more specific, consider an $n$-dimensional vector space $V$
over a finite field $\F=\F_q$ of order $q$ and let $D$ be a 
differential on $V$, i.e., a linear operator $D\colon V\to V$ with
$D^2=0$. We are interested in the probability $p_r(q,n)$ with which a
chain complex $(V,D)$ has homology $\ker D/\im D$ of a given dimension $r$
for fixed $n$ and $q$. The differential $D$ is uniformly distributed
and $p_r(q,n)$ is simply the ratio $c_r(q,n)/c(q,n)$, where $c_r(q,n)$
is the number of complexes with $r$-dimensional homology (see Theorem
\ref{thm:number}) and $c(q,n)$ is the number of all complexes.

We mainly focus on large complexes, i.e., on the limits as $q$ or $n$
go to infinity. Clearly, $r$ and $n$ must have the same parity and
we separately analyze the asymptotic behavior of the sequence
$p_0(q,n), p_2(q,n),\ldots $, where $n$ is even, and
the sequence $p_1(q,n), p_3(q,n),\ldots $ for $n$ odd. 

As $q\to\infty$ with $n$ fixed, the probability concentrates in the
lowest possible dimension, i.e., $p_0(q,n)\to 1$ or $p_1(q,n)\to 1$
depending on the parity of $n$, while $p_r(q,n)\to 0$ for $r>1$. This
is consistent with the observation that over $\C$ and even $\R$ (see
Lemma \ref{lemma:D}) a generic complex has $0$- or $1$-dimensional
homology, i.e., that such complexes form the highest dimensional
stratum in the variety of all $n$-dimensional complexes. Indeed, one
can expect the probability distributions for large $q$ to approximate
the generic situation in zero characteristic. We do not know, however,
if the density functions converge in any sense as $q\to\infty$ to some
probability density on the variety of $n$-dimensional complexes over,
e.g., $\R$.

When $q$ and $r$ are fixed and $n\to\infty$ through either even or odd
integers depending on the parity of $r$, the situation is more
subtle. In this case, all limit probabilities
$p_r(q)=\lim_{n\to\infty}p_r(q,n)$ are positive. However, the
sequences $p_0,p_2,\ldots$ and $p_1,p_3,\ldots$ are super-exponentially
decreasing and for a large $q$ all terms in these sequences but the
first one are very close to zero while the first is then, of course,
close to 1. When $q=2$ and $r$ is even, we have $p_0\approx 0.6$,
$p_2\approx 0.4$, $p_4\approx 0.0075$ and other terms are very
small. We explicitly calculate the ratios $p_r(q)/p_0(q)$ and
$p_r(q)/p_1(q)$ and $p_0$ and $p_1$ in Theorem \ref{thm:prob}.

The proofs of these facts are elementary and quite simple. However, we
have not been able to find our results in the literature or any
probability calculations in this basic case where chain
complexes are stripped of all additional structures including a
grading. In contrast, random complexes of geometrical origin and
underlying random geometrical and topological objects have been
studied extensively and from various perspectives. Among such random
objects are, for instance, random simplicial complexes of various
types (see \cite{ALLM,BK,CF14,CF15,Ka,Me,MW,PS,YSA} and references
therein) and random Morse functions (see, e.g.,
\cite{Ar06,Ar07,CHKW,Ni}).

These works utilize several models of randomness all of which appear
to be quite different from the one, admittedly rather naive, used
here. This makes direct comparison difficult. One way to interpret our
result is that, for a large complex, sufficiently non-trivial homology
is indicative of the presence of some structure, a constraint limiting
randomness. Note that such a structure can be as simple as a
$\Z$-grading confined to a fixed range of degrees. A dimensional
constraint of this type is usually inherent in geometrical 
complexes, and it would be interesting to analyze its effect (if any)
on the probability distribution in our purely algebraic
setting. Another consequence of the result is that the assertion that
a complex has large homology carries more information than the
assertion that it has small homology.

The main motivation for our setting comes from Hamiltonian Floer
theory for closed symplectic manifolds; see, e.g., \cite{Sa} and
references therein. A Hamiltonian diffeomorphism is the time-one map
of the isotopy generated by a time-dependent Hamiltonian. To such a
diffeomorphism one can associate a certain complex, called the Floer
complex, generated by its fixed points or, equivalently, the
one-periodic orbits of the isotopy. Hence the dimension of the Floer
homology gives a lower bound for the number of one-periodic orbits.
The homology is independent of the Hamiltonian
diffeomorphism. In addition, one can fix the free homotopy class of
the orbits. (This construction is similar to Morse theory and, in
fact, Floer theory is a version of Morse theory for the action
functional.)

In many instances, e.g., often generically or for all symplectic
manifolds with vanishing first Chern class such as tori, the
dimension of the Floer complex grows with the order of iteration of
the diffeomorphism; see \cite{GG:CC-survey}. In other words, the
complex gets larger and larger as time in this discrete dynamical
system grows.  Moreover, the differential in the complex is usually
impossible to describe explicitly, and hence it makes sense to compare
the behavior of the complex and its homology with the generic or
random situation.  The Floer homology for contractible periodic orbits
is isomorphic to the homology of the underlying manifold. Therefore, by
our result, even though the Floer complex appears to be very ``noisy''
for large iterations and random on a bounded action scale, it has
large homology groups and is actually very far from random. For
non-contractible orbits, the dimension of the Floer complex is also
known to grow in many settings; see \cite{GG:nc,Gu:nc}. However, in this
case the Floer homology is zero and the complex may well be close to
random.  Note also that in some instances the Floer complex is
$\Z$-graded, but the grading is not supported within any specific
interval of degrees. Moreover, in contrast with geometrical random
complexes, the grading range of the Floer homology usually grows with
the order of iteration (\cite{SZ}), and while it is not clear how to
correctly account for an unbounded grading in a random model, such a
grading is unlikely to affect the probability distribution.

One aspect of Floer theory which is completely ignored in our model is
the action filtration. This filtration is extremely important and, in
particular, it allows one to treat Floer theory in the context of
persistent homology and topological data analysis; see
\cite{Ca,Gh}. This connection has recently been explored in
\cite{PoS,UZ}.  However, it is not entirely clear how
to meaningfully incorporate the action filtration into our model.

\subsection*{Acknowledgments}
The authors are grateful to Robert Ghrist, Ba\c sak G\"urel, Jiang-Hua
Lu, Roy Meshulam and Leonid Polterovich for useful discussions and
comments.  The authors would also like to thank the referee for
pointing out \cite[Lemma 5]{Ko} to them. A part of this work was carried out
while the second author was visiting the Simons Institute for the
Theory of Computing and he would like to thank the institute for its
warm hospitality.

\section{Main Results}
\label{sec:results} 

Let, as in the introduction, $(V,D)$ be an ungraded $n$-dimensional
chain complex with differential $D$ over a finite field $\F=\F_q$ of order
$q$. In other words, $V=\F^n$ and $D$ is a linear operator on $V$ with
$D^2=0$. We denote by $c(q,n)$ the number of such complexes, i.e., the
number of differentials $D$. The dimension $r$ of the homology
$\ker D/\im D$ has the same parity as $n$ and we let $c_r(q,n)$ be the
number of complexes with homology of dimension $r$. (In what follows,
we always assume that $r$ and $n$ have the same parity.) 
Clearly,
$$
c(q,n)=c_0(q,n)+c_2(q,n)+\cdots+c_{n}(q,n)
$$ 
when $n$ is even and 
$$
c(q,n)=c_1(q,n)+c_3(q,n)+\cdots+c_{n}(q,n)
$$ 
when $n$ is odd.

Furthermore, denote by 
$$
p_r(q,n)=\frac{c_r(q,n)}{c(q,n)}
$$
the probability (with respect to the
uniform distribution) of a complex to have $r$-dimensional
homology. Our main result describes the behavior of $p_r(q,n)$ as the
size of the complex, i.e., $q$ or $n$, goes to infinity.

\begin{Theorem}
\label{thm:prob}
Let $p_r(q,n)$ be as above.
\begin{itemize}
\item[(i)] 
For a fixed $n$, we have 
$$
\lim_{q\to\infty} p_r(q,n)=0\textrm{ when $r>1$},
$$
and $p_0(q,n)\to 1$ when $n$ is even and $p_1(q,n)\to 1$ when $n$ is
odd as $q\to\infty$.
\item[(ii)] 
For a fixed $q$ and $r$, the limits
$$
p_r(q)=\lim_{n\to\infty} p_r(q,n) 
$$
exist and $0<p_r(q)<1$ for all $q$ and $r$. Furthermore, when $r\geq 2$ is even, we have
\begin{equation}
\label{eq:lim}
\frac{p_r(q)}{p_0(q)}=\frac{q^{r/2}}{\prod\limits_{j=1}^r(q^j-1)}
\end{equation}
and
\begin{equation}
\label{eq:S}
p_0(q)=\frac{1}{1+S}, \textrm{ where } S=\sum_{k=1}^\infty
\frac{q^{k}}{\prod\limits_{j=1}^{2k}(q^j-1)}.
\end{equation}
When $r\geq 3$ is odd, 
$$
\frac{p_r(q)}{p_1(q)}=\frac{(q-1)q^{(r-1)/2}}{\prod\limits_{j=1}^r(q^j-1)}
$$
and
$$
p_1(q)=\frac{1}{1+S'}, \textrm{ where } S'=(q-1)\sum_{k=1}^\infty
\frac{q^{k}}{\prod\limits_{j=1}^{2k+1}(q^j-1)}.
$$
\end{itemize}
\end{Theorem}

The proof of this theorem is based on an explicit calculation of
$c_r(q,n)$. To state the result, denote by $\GL_k(q)$ the general
linear group of $k\times k$ invertible matrices over $\F_q$ and recall that
$$
|\GL_k(q)|=q^{k(k-1)/2}\prod_{j=1}^k(q^j-1).
$$  
Then we have the following particular case of \cite[Lemma 5]{Ko}.

\begin{Theorem}[Kovac, \cite{Ko}]
\label{thm:number}
Let as above $c_r(q,n)$ be the number of $n$-dimensional complexes
over $\F_q$ with homology of dimension $r$. Then
\begin{equation}
\label{eq:number}
c_r(q,n)=\frac{|\GL_n(q)|}{|\GL_m(q)|\cdot |\GL_r(q)|\cdot
  q^{2mr+m^2}}\, ,
\end{equation}
where $2m+r=n$.
\end{Theorem}
Even though this result is not new, for the sake of completeness we
include its proof, which is very simple and short, in the next
section.

\begin{Remark}
  We do not have simple expressions for the probabilities $p_r(q,n)$
  and the total number of complexes $c(q,n)$. However, when $q=2$, the
  differentials $D$ are in one-to-one correspondence with involutions
  of $\F_2^n$. (An involution necessarily has the form $I+D$ and, as is
  easy to see, different differentials $D$ give rise to different
  involutions.) Hence, $c(2,n)$ is equal to the number of
  involutions. This number is expressed in \cite{FV} via a generating
  function and an asymptotic formula for $c(q,n)$ has been recently
  obtained in \cite{FGS}. It is possible that at least when $q=2$ our
  probability formulas can be further simplified using the results
  from those papers.
\end{Remark}

\section{Proofs}
\label{sec:proofs}

The proof of Theorem \ref{thm:number} is based on the observation that
the differential in a finite-dimensional complex over any field $\F$
can be brought to its Jordan normal form or, equivalently, a complex
over $\F$ can be decomposed into a sum of elementary complexes, i.e.,
into a sum of two-dimensional complexes with zero homology and
one-dimensional complexes. To be more precise, we have the following
elementary observation.

\begin{Lemma}
\label{lemma:D}
Let $V$ be a finite-dimensional vector space over an arbitrary field
$\F$ and let $D\colon V\to V$ be an operator with $D^2=0$. Then, in some basis,
$D$ can be written as a direct sum of $1\times 1$ and $2\times 2$
Jordan blocks with zero eigenvalues.
\end{Lemma}

When $\F$ is algebraically closed, this follows immediately from the
Jordan normal form theorem. Hence, the emphasis here is on the fact
that the field $\F$ is immaterial. For the sake of completeness, we
outline a proof of the lemma. 

\begin{proof} Let us pick an arbitrary basis $\{e_1,\ldots, e_m\}$ of
  $\im D$ and extend it to a basis of $\ker D\supset \im D$ by adding elements
  $\{f_1,\ldots,f_r\}$. Furthermore, pick arbitrary vectors $e'_i$ with
  $De'_i=e_i$. Then $\{e'_1, e_1, \ldots, e'_m, e_m, f_1, \ldots, f_r\}$
  is the required basis of $V$.
\end{proof}

\begin{proof}[Proof of Theorem \ref{thm:number}]
  Let $D$ be a differential on an $n$-dimensional vector space $V$
  over a finite field $\F=\F_q$. Assume that the homology of the
  complex $(V,D)$ is $r$-dimensional. By Lemma \ref{lemma:D}, $D$ is
  conjugate to the map $D_r$ which is the direct sum of $r$
  $1\times 1$ zero blocks and $m$ $2\times 2$ Jordan blocks with zero
  eigenvalues, where $2m+r=n$.

  Let $C_r$ be the centralizer of $D_r$ in $\GL_n(q)$.  The complexes with
  $r$-dimensional homology are in one-to-one correspondence with
  $\GL_n(q)/C_r$. Thus, to prove \eqref{eq:number}, it suffices to
  show that
\begin{equation}
\label{eq:center}
|C_r|=|\GL_m(q)|\cdot |\GL_r(q)|\cdot
  q^{2mr+m^2} .
\end{equation}
The elements of $C_r$ are $n\times n$ invertible matrices
$X\in \GL_n(q)$ commuting with $D_r$. In what follows, it is
convenient to work with the basis
$e_1, \ldots, e_m, f_1, \ldots, f_r, e'_1,\ldots,e'_m$ in the notation
from the proof of Lemma \ref{lemma:D}. Thus we can think of $X$ as a
$3\times 3$-block matrix with $m\times m$ block $X_{11}$, the block
$X_{12}$ having size $m\times r$, and $X_{13}$ being again
$m\times m$, etc. In the same format, $D_r$ is then the matrix with
only one non-zero block. This is the top-right corner
$m\times m$-block, which is $I$. Then, as a straightforward
calculation shows, the commutation relation $XD_r=D_rX$ amounts to the
conditions that $X_{11}=X_{33}$, and $X_{21}=0$, $X_{31}=0$ and
$X_{32}=0$. In particular, $X$ is an upper block-triangular
matrix. Hence, $X$ is invertible if and only if $X_{11}=X_{33}$ and
$X_{22}$ are invertible. There are no constraints on the remaining blocks
$X_{12}$, $X_{13}$ and $X_{23}$. Now \eqref{eq:center} follows.
\end{proof}

\begin{proof}[Proof of Theorem \ref{thm:prob}]
Throughout the proof, we assume that $r$ and $n$ are even. The case
where these parameters are odd can be handled in a similar fashion.

As the first step, we express $p_r(q,n)/p_0(q,n)$
explicitly. Clearly,
$$
\frac{p_r(q,n)}{p_0(q,n)}=\frac{c_r(q,n)}{c_0(q,n)}=\frac{|C_0|}{|C_r|}.
$$
Using \eqref{eq:number} or \eqref{eq:center} and tidying up the
resulting expression, we have
\begin{equation*}
\begin{split}
\frac{p_r(q,n)}{p_0(q,n)} &=\frac{q^{n^2/4}\cdot
  q^{n(n/2-1)/4}\cdot \prod_{j=1}^{n/2}(q^j-1)}
{q^{2mr+m^2}\cdot q^{m(m-1)/2}\cdot q^{r(r-1)/2}\cdot\prod\limits_{j=1}^{r}(q^j-1)
\cdot\prod\limits_{j=1}^{m}(q^j-1)}
\\
&=\frac{\prod\limits_{j=m+1}^{m+r/2}(q^j-1)}
{q^{mr/2+r(r/2-1)/4}\cdot \prod\limits_{j=1}^{r}(q^j-1)
},
\end{split}
\end{equation*}
where as above $n=2m+r$ and $r\geq 2$, which we can then rewrite as
\begin{equation}
\label{eq:ratio}
\frac{p_r(q,n)}{p_0(q,n)}=\frac{q^{r/2}}
{\prod\limits_{j=1}^{r}(q^j-1)
}\cdot\prod\limits_{j=1}^{r/2}\left(1-\frac{1}{q^{m+j}}\right).
\end{equation}
Now it is clear that 
$$
\frac{p_r(q,n)}{p_0(q,n)}\sim q^{-r^2/2}\textrm{ as } q\to\infty
$$
with $r\geq 2$ and $n$ fixed. In particular, this ratio goes to zero as $q\to
\infty$. The number of the terms in the sum
$$
\sum_j p_j(q,n)=1
$$
with $j$ ranging through even integers from $0$ to $n$ is equal to
$n/2+1$ and thus this number is independent of $q$. Hence,
$p_0(q,n)\to 1$ and $p_r(q,n)\to 0$ when $r\geq 2$ as
$q\to\infty$. This proves the first assertion of the theorem.

To prove the second part, first note that by \eqref{eq:ratio}
\begin{equation}
\label{eq:lim-ratio}
\frac{p_r(q,n)}{p_0(q,n)}\to \frac{q^{r/2}}{\prod\limits_{j=1}^r(q^j-1)}
\end{equation}
as $m\to \infty$ or, equivalently, $n\to \infty$ with $r$ and $q$ fixed.

Furthermore, in a similar vein, it is not hard to show that
$$
\sum_{r>0} \frac{p_r(q,n)}{p_0(q,n)}\to S:=\sum_r
\frac{q^{r/2}}{\prod\limits_{j=1}^r(q^j-1)}\textrm{ as } n\to\infty,
$$
where, on the left, the sum is taken over all even integers from $2$ to $n$
and, on the right, the sum is over all even integers $r\geq 2$. Therefore, letting $n\to\infty$
in the identity
$$
1+\sum_{r>0} \frac{p_r(q,n)}{p_0(q,n)}=\frac{1}{p_0(q,n)},
$$
we conclude that the limit $p_0(q)=\lim_{n\to\infty}p_0(q,n)$ exists
and $p_0(q)=1/(1+S)$, which proves \eqref{eq:S}. Now, by
\eqref{eq:lim-ratio}, the limits $p_r(q)=\lim_{n\to\infty}p_r(q,n)$
for $r\geq 2$ also exist, and hence \eqref{eq:lim} holds. This
completes the proof of the theorem.
\end{proof}

\begin{Remark}
The sequence $c_r(q,n)$ is decreasing as a function of $r$. This
readily follows from \eqref{eq:ratio}.
\end{Remark}

\end{document}